\documentclass[a4paper,10pt]{article}
\usepackage{amsmath}
\usepackage{amssymb}
\usepackage{amscd}
\usepackage{amsthm}
\newtheorem{theorem}{Theorem}
\newtheorem{lemma}{Lemma}
\newtheorem{definition}{Definition}
\newtheorem{remark}{Remark}

\newtheorem{result}{Result}

\newtheorem{corollary}{Corollary}

\def\F{\mathbb{F}}

\date{\today}
\title{On the code generated by the incidence matrix of points and $k$-spaces in $PG(n,q)$ and its dual}
\author{M. Lavrauw\and L. Storme \and G. Van de Voorde \thanks{This author's research was supported by the                 Institute for the Promotion of Innovation through Science
and Technology in Flanders (IWT-Vlaanderen).} }

\begin{document}

\maketitle

\begin{abstract} In this paper, we study the $p$-ary linear code $C_{k}(n,q)$, $q=p^h$, $p$ prime, $h\geq 1$, generated by the incidence matrix of points and $k$-dimensional spaces in $PG(n,q)$. For $k\geq n/2$, we link codewords of $C_{k}(n,q)\setminus C_{k}(n,q)^\perp$  of weight smaller than $2q^k$ to $k$-blocking sets. We first prove that such a $k$-blocking set is uniquely reducible to a minimal $k$-blocking set, and exclude all codewords arising from small linear $k$-blocking sets. For $k<n/2$, we present counterexamples to lemmas valid for $k\geq n/2$.
Next, we study the dual code of $C_k(n,q)$ and present a lower bound on the weight of the codewords, hence extending the results of Sachar \cite{sachar} to general dimension.
\end{abstract}

\section{Introduction}
Let $PG(n,q)$ denote the $n$-dimensional projective space
over the finite field  $\F_q$ with $q$ elements, where $q=p^h$,  $p$ prime, $h\geq 1$, and let $V(n+1,q)$ denote the underlying vector space. Let $\theta_n$ denote the number of points in $PG(n,q)$, i.e., $\theta_n=(q^{n+1}-1)/(q-1)$.
A {\em blocking set} of $PG(n,q)$ is a
set $K$ of points such that each hyperplane of $PG(n,q)$ contains at least one point
of $K$. A blocking set $K$ is called {\em trivial}
 if it contains a line of $PG(n,q)$. 
These blocking sets are also called 
{\em 1-blocking sets} in \cite{AB:80}. In general, a \emph{$k$-blocking
set} $K$ in $PG(n,q)$ is a set of points  such that any $(n-k)$-dimensional
subspace intersects $K$. A $k$-blocking set $K$ is called {\em trivial} when a $k$-dimensional subspace is contained in 
$K$. 
The smallest non-trivial $k$-blocking sets are characterized as cones
with a $(k-2)$-dimensional vertex $\pi_{k-2}$ and a non-trivial 1-blocking set
of minimum cardinality in a plane, skew to $\pi_{k-2}$, of $PG(n,q)$ as base curve \cite{AB:80,UH:98}.
If an $(n-k)$-dimensional space contains
exactly one point of a $k$-blocking set $K$ in $PG(n,q)$, it is called a {\em tangent $(n-k)$-space} to $K$, and a point
$P$ of
$K$ is called {\em essential} when it belongs to a tangent $(n-k)$-space of $K$.
A $k$-blocking set $K$ is called {\em minimal} when no proper subset of $K$
is also a $k$-blocking set, i.e., when each point of $K$ is essential.

%A line intersecting $K$ in at least two points is called a {\em secant}
%of  $K$. If the line contains $i$ points of $K$, it is
%called an {\em $i$-secant}. The 1-secants are also called {\em tangent lines}
%to $K$.

A lot of attention has been paid to blocking sets in the Desarguesian plane
$PG(2,q)$, and to $k$-blocking sets in $PG(n,q)$. It follows from results of Sziklai \cite{sziklai}, Sz\H{o}nyi \cite{TS:97}, and Sz\H{o}nyi and Weiner \cite{TS:??} that every minimal $k$-blocking set $K$ in $PG(n,q)$, $q=p^h$, $p$ prime, $h\geq 1$, of size smaller than $3(q^{n-k}+1)/2$, intersects every subspace in zero or in $1\pmod{p}$ points. If $e$ is the largest integer such that $K$ intersects every space in zero or $1\pmod{p^e}$ points, then $e$ is a divisor of $h$. This implies, for instance, that the cardinality of a minimal blocking set, of size smaller
than $3(q+1)/2$, in $PG(2,q)$ can only lie in a number of intervals, each of
which corresponds to a divisor $e$ of $h$. 
%

%A number of articles \cite{ SMB:96, AAB:87,AAB:77}
%studied   the problem of finding the size of
%the second smallest minimal blocking set in $PG(2,q^2)$. The smallest
%blocking set in $PG(2,q^2)$ is a Baer subplane \cite{AAB:70}.
% Blokhuis, Storme, and Sz\H onyi \cite{AB:97} proved that $|K| \geq q^2+q^{4/3}+1$ if $q=p^h, h > 1,p \geq 5$,
%for a blocking set not containing a Baer subplane, and Sz\H onyi 
%\cite{TS:97}
%proved
%that, in $PG(2,p^2)$, $p$ prime,
%$|K| =3(p^2+1)/2$ for the smallest minimal blocking sets
%not containing a Baer subplane.

% Heim introduced in \cite{UH:97} the concept of {\em proper blocking set}.
%A proper blocking set is a blocking set $K$ of $\Sigma$ such that
%no hyperplane of $\Sigma$ intersects $K$ into a blocking set of that hyperplane.
We define the incidence matrix~$A = (a_{ij})$ 
of points and $k$-spaces in the projective space $PG(n,q)$, $q=p^h$,  $p$~prime, $h\geq 1$, 
as the matrix whose rows are 
indexed by the $k$-spaces of $PG(n,q)$ and whose columns are indexed
by the points of $PG(n,q)$, and 
with entry
$$ 
a_{ij} = \left\{
\begin{array}{ll}
1 & \textrm{if point $j$ belongs to $k$-space $i$,}\\
0 & \textrm{otherwise.}
\end{array} 
\right.
$$
The $p$-ary linear code $C$ of points and $k$-spaces of $PG(n,q)$, $q=p^h$, $p$ prime, $h\geq 1$,  is 
the $\mathbb{F}_p$-span of the rows of the  incidence matrix $A$. From now on, we denote this code by $C_k$, or, if we want to specify the dimension and order of the ambient space, by $C_k(n,q)$. The {\em support} of a codeword $c$, denoted by $supp(c)$, is the set of all non-zero positions of $c$. The {\em weight} of $c$ is the number of non-zero positions of $c$ and is denoted by $wt(c)$. Often we identify the support of a codeword with the corresponding set of points of $PG(n,q)$. We let $c_P$ denote the symbol of the codeword $c$ in the coordinate position corresponding to the point $P$, and let $(c_1,c_2)$ denote the scalar product in $\F_p$ of two codewords $c_1, c_2$ of $C$. Furthermore, if $T$ is a subspace of $PG(n,q)$, then the incidence vector of this subspace is also denoted by $T$.
The dual code $C^\bot$ is the set of all vectors orthogonal to all codewords of $C$, hence
$$C_k^\bot=\{ v\in V(\theta_n,p) || (v,c)= 0,\ \forall c\in C_k\}.$$
This means that for all $c\in C_k^{\bot}$ and all $k$-spaces $K$ of $PG(n, q)$, we have $(c,K)=0$.
In \cite{art}, the $p$-ary linear code $C_{n-1}(n,q)$, $q=p^h$, $p$ prime, $h\geq 1$, was discussed. The main goal of this paper is to prove similar results for the $p$-ary linear code $C_k(n,q)$ defined by the incidence matrix of points and $k$-spaces of $PG(n,q)$, $q=p^h$, $p$ prime, $h\geq 1$. 
More precisely, in \cite{art}, the following results are proven.

\begin{result} (see also \cite[Proposition 5.7.3]{AK}) The minimum weight codewords of $C_{n-1}(n,q)$ are the scalar multiples of the incidence vectors of the hyperplanes.
\end{result}
\begin{result} There are no codewords with weight in the interval $]\theta_{n-1},2q^{n-1}[$ in $C_{n-1}(n,q)$, if $q$ is prime, or if $q=p^2$, $p>11$ prime.
\end{result}
\begin{result}
The only possible codewords of $C_{n-1}(n,q)$, with weight in the interval
$]\theta_{n-1},2q^{n-1}[$, are the scalar multiples of non-linear minimal blocking sets.
\end{result}
\begin{result} \label{th11}
The minimum weight of $C_{n-1}(n,q)\cap C_{n-1}(n,q)^\bot$ is equal to $2q^{n-1}$.
\end{result}
\begin{result} 
If $c$ is a codeword of $C_{n-1}(n,q)^\bot$ of minimal weight, then $supp(c)$ is contained in a plane of $PG(n,q)$.
\end{result}

Theorem \ref{th5}(2) and Theorem \ref{?} extend Result 1 and the first part of Result 2 to general dimension. However, the generalization of the second part of Result 2 in Theorem \ref{th20} and the generalization of Result 3 in Theorem \ref{th5}(1) are weaker, due to the lack of a generalization of Result 4 in the case where $q$ is not a prime.
In Theorem \ref{th3}, Result 5 is generalized.

In the study of codewords $c\in C_k(n,q)$ of weight smaller than $2q^k$, we distinguish the cases $c\in C_k(n,q)\setminus C_k(n,q)^{\perp}$ and $c\in C_k(n,q)\cap C_k(n,q)^{\perp}$. In the first case, for $k\geq n/2$, $supp(c)$ defines a $k$-blocking set of $PG(n,q)$. We eliminate the small linear $k$-blocking sets as possible codewords, if $k\geq n/2$. One of the results we need regarding  $k$-blocking sets, is the unique reducibility property of  $k$-blocking sets, of size smaller than $2q^k$, to a minimal $k$-blocking set.
We derive this property in the next section.

\section{A unique reducibility property for $k$-blocking sets in $PG(n,q)$ of size smaller than $2q^k$}

In \cite{TS:97}, algebraic curves are associated to blocking sets in $PG(2,q)$, in order to prove the following result.
\begin{result}\cite[Sz\H onyi]{TS:97}
If $K$ is a blocking set in $PG(2,q)$ of cardinality $|K| \leq 2q$, then $K$ can be reduced in a unique way to a minimal blocking set.
\end{result}
In this section, we extend this result to general $k$-blocking sets in $PG(n,q)$, $n\geq 3$, by associating an algebraic hypersurface to a blocking set in $PG(n,q)$.

Let $K$ be a blocking set in $PG(n,q)$, $n \geq 3$, with  $|K| 
\leq 2q-1$. Suppose that the coordinates
of the points are $(x_0,\ldots,x_n)$, where $X_n=0$ defines the hyperplane
at infinity $H_{\infty}$, and let $U$ be the set of affine points of $K$. Let $|K|=q+k+N$, $N \geq 1$, where $N$ is the number of points of
$K$ in $H_\infty$. Furthermore we assume that $(0,\ldots,0,1,0)\in K$. 
The hyperplanes not passing through $(0,\ldots,0,1,0)$
have equations $m_0X_0+\cdots+m_{n-2}X_{n-2}-X_{n-1}+b=0$
and they intersect $H_\infty$ in the $(n-2)$-dimensional
space $X_n=m_0X_0+\cdots +m_{n-2}X_{n-2}-X_{n-1}=0$. We call
the $(n-1)$-tuple $\bar{m}=
(m_0,\ldots,m_{n-2})$ the {\em slope} of the hyperplane.
We also identify a slope $\bar{m}$ with the corresponding subspace $X_n=m_0X_0+\cdots +m_{n-2}X_{n-2}-X_{n-1}=0$ of dimension $n-2$ at infinity.

\begin{definition}
Define the {\em R\'edei polynomial} of $U$ as
\begin{eqnarray*}
H(X,X_0,\ldots,X_{n-2})& = & \prod_{(a_0,\ldots,a_{n-1})\in U}(X+a_0X_0+\cdots+a_{n-2}X_{n-2}-a_{n-1})\\
& =& 
X^{q+k}+h_1(X_0,\ldots,X_{n-2})X^{q+k-1}+\cdots+\\&&h_{q+k}(X_0,\ldots,X_{n-2}).
\end{eqnarray*}

\end{definition}
For all $j=1,\ldots,q+k$, $\deg h_j \leq j$.
For simplicity of notations, we will also write $H(X,X_0,\ldots,X_{n-2})$
as $H(X,\bar{X})$.

\begin{definition}
Let $C$ be the affine hypersurface, of degree $k$, of $AG(n,q)$,
 defined by
\[f(X,\bar{X})=X^k+h_1(\bar{X})X^{k-1}+\cdots+h_k(\bar{X})=0.\]
\end{definition}

\begin{theorem} \label{th:2}
(1) For a fixed slope $\bar{m}$ defining an $(n-2)$-dimensional subspace at
infinity not containing a point of $K$, the polynomial
$X^q-X$ divides $H(X,\bar{m})$.
Moreover, if $k < q-1$, then $H(X,\bar{m})/(X^q-X)=f(X,\bar{m})$
and $f(X,\bar{m})$ splits into linear factors over $\F_q$.

(2) For a fixed slope $\bar{m}=(m_0,\ldots,m_{n-2})$, the element $x$ is an
$r$-fold root of $H(X,\bar{m})$ if and only if the hyperplane with
equation $m_0X_0+\cdots+m_{n-2}X_{n-2}-X_{n-1}+x=0$ intersects $U$
in exactly $r$ points.
 
(3) If $k< q-1$ and $\bar{m}$ defines an $(n-2)$-dimensional subspace at infinity not 
containing a point of $K$, such that the line $X_0=m_0,\ldots,X_{n-2}=m_{n-2}$
intersects $f(X,\bar{X})$ at $(x,m_0,\ldots,m_{n-2})$ with multiplicity $r$,
then the hyperplane  with equation $m_0X_0+\cdots+m_{n-2}X_{n-2}-X_{n-1}+x=0$
intersects $K$ in exactly $r+1$ points.
\end{theorem}
\begin{proof}
(1) For every $X=b$, the hyperplane $m_0X_0+\cdots+m_{n-2}X_{n-2}-X_{n-1}+b=0$
contains a point $(a_0,\ldots,a_{n-1})$ of $U$. So $X-b$ is a factor of
$H(X,\bar{m})$.

If $k <q-1$, then $H(X,\bar{m})=X^{q+k}+h_1(\bar{m})X^{q+k-1}
+\cdots+h_{q+k}(\bar{m})=(X^k+h_1(\bar{m})X^{k-1}+\cdots+h_k(\bar{m}))(X^q-X)
=f(X,\bar{m})(X^q-X)$.

Since $H(X,\bar{m})$ splits into linear factors over $\F_q$,
this is also true for $f(X,\bar{m})$.

(2) The multiplicity of a root $X=x$ is the number of linear factors
in the product defining $H(X,\bar{m})$ that vanish at $(x,\bar{m})$. This
is the number of  points of $U$ lying on  the hyperplane
$m_0X_0+\cdots+m_{n-2}X_{n-2}-X_{n-1}+x=0$.

(3) The slope $(m_0,\ldots,m_{n-2})$ defines an $(n-2)$-dimensional
subspace at infinity not containing a point of $K$.
If the intersection multiplicity is $r$, then $x$ is an $(r+1)$-fold root
of $H(X,\bar{m})$. Hence, the result follows from (1) and (2).
\end{proof}

\begin{remark}\label{(n-2)-space}
By induction on the dimension, one can construct an $(n-2)$-dimensional subspace $\alpha $ skew to $K$.  Since  $|K| \leq 2q-1$, $K$ has a tangent hyperplane because all hyperplanes through
$\alpha$ must contain at least one point of $K$.\end{remark}

Assume that $X_n=0$ is a tangent hyperplane
to $K$ in the point $(0,\ldots,0,1,0)$.
The following theorem links the problem of minimality of the  blocking
set $K$ to that of the problem of finding  linear factors
of the affine hypersurface $C:f(X,\bar{X})=0$.

\begin{theorem}\label{th:6}
(1) If a point $P=(a_0,\ldots,a_{n-1}) \in U$ is not essential,
then the linear factor $a_0X_0+\cdots+a_{n-2}X_{n-2}-a_{n-1}+X$
divides $f(X,\bar{X})$.

(2) If the linear factor $X+a_0X_0+\cdots+a_{n-2}X_{n-2}-a_{n-1}$ divides
$f(X,\bar{X})$, then $P=(a_0,\ldots,a_{n-1}) \in U$ and this point is not essential.
\end{theorem}
\begin{proof}
(1) Consider an arbitrary slope $\bar{m}=(m_0,\ldots,m_{n-2})$. For this
slope $\bar{m}$, there are at least two points of $K$ in
the hyperplane $m_0X_0+\cdots+m_{n-2}X_{n-2}-X_{n-1}+b=0$ through
$(a_0,\ldots,a_{n-1})$. Hence, by Theorem \ref{th:2},  the hyperplane
$\pi:a_0X_0+\cdots+a_{n-2}X_{n-2}-a_{n-1}+X=0$ shares the
point $(X,\bar{X})=(a_{n-1}-(a_0m_0+\cdots+a_{n-2}m_{n-2}),m_0,\ldots,m_{n-2})$ with
$C$.
Suppose that $a_0X_0+\cdots+a_{n-2}X_{n-2}-a_{n-1}+X$ does not divide
$f(X,\bar{X})$, and let $R$ be a point of the hyperplane $\pi$ not lying in $C$.

There are $q^{n-2}+\cdots+q+1$ lines through $R$ in the hyperplane $\pi$, and none of them is contained in $C$ since $R \not\in C$. Since such lines
contain at most $k$ points of $C$, $\pi$ contains
at most $k(q^{n-2}+\cdots+q+1) <(q-1)(q^{n-2}+\cdots+q+1)=q^{n-1}-1$
points of $C$. This is a contradiction since the number of possibilities for $\bar{m}$ is $q^{n-1}-1$, and each slope corresponds to a distinct point of $\pi \cap C$.

(2) If this linear factor divides $f(X,\bar{X})$, then for all $\bar{m}=(m_0,\ldots,m_{n-2})$, the hyperplane with slope $\bar{m}$
through $(a_0,\ldots,a_{n-1})$ intersects $U $ in at least two points
(Theorem  \ref{th:2} (3)). Here, we use that $X_n=0$ is a tangent hyperplane to $K$ in the point $(0,\ldots,0,1,0)$, so $\bar{m}$ defines an $(n-2)$-dimensional subspace at infinity not containing a point of $K$.

Suppose that $(a_0,\ldots,a_{n-1}) \not\in U$. By induction, it is possible to
prove that there is a subspace $\pi$
of dimension $n-2$ passing through 
$(a_0,\ldots,a_{n-1})$ and containing no  points of $K$ (cf. Remark \ref{(n-2)-space}).
Consider all hyperplanes through $\pi$. One of them passes through $(0,\ldots,0,1,0)$; the other ones contain at least two points of $K$. So
$|K| \geq 2q+1$, which is false.

Hence, $P=(a_0,\ldots,a_{n-1}) \in U$. Since all hyperplanes through
$P$, including those through $(0,\ldots,0,1,0)$, contain at least
two points of $K$, the point $P$ is not essential.
\end{proof}

\begin{corollary} \label{uniHyp} A blocking set $B$ of size smaller than $2q$ in $PG(n,q)$ is uniquely reducible to a minimal blocking set.
\end{corollary}
\begin{proof} The non-essential points of $B$ correspond to the linear factors over $\mathbb{F}_q$ of the polynomial $f(X,\bar{X})$, and this polynomial is uniquely reducible.
\end{proof}
We will extend this unique reducibility property to blocking sets with respect to $k$-blocking sets.

\begin{theorem} \label{uni}A $k$-blocking set in $PG(n,q)$ of size smaller than $2q^{k}$  is uniquely reducible to a minimal $k$-blocking set.
\end{theorem}
\begin{proof} Embed $PG(n,q)$ in $PG(n,q^{k})$. Let $\pi$ be a hyperplane of $PG(n,q^{k})$.  Let $\pi^{q^i}=\lbrace (x_0^{q^i},\ldots,x_{n}^{q^i})\vert \vert (x_0,\ldots,x_{n})\in \pi \rbrace$. The space $\pi \cap \pi^q \cap \pi^{q^2} \cap \cdots \cap \pi^{q^{k-1}}$ is the intersection of $\pi$ with $PG(n,q)$. Since it is the intersection of $k$ (not necessarily distinct) hyperplanes, it has dimension at least $n-k$. This implies that a $k$-blocking set $B$ in $PG(n,q)$ is also a $1$-blocking set in $PG(n,q^{k})$. In Corollary \ref{uniHyp}, it is proven that this latter blocking set is uniquely reducible to a minimal $1$-blocking set $B'$ in $PG(n,q^{k})$. Since every $(n-k)$-dimensional space $\Pi$ in $PG(n,q)$ can be extended to a hyperplane in $PG(n,q^{k})$ that intersects $PG(n,q)$ only in $\Pi$ (straightforward counting), it is easy to see that the minimal blocking set $B'$ in $PG(n,q^{k})$ is the unique minimal $k$-blocking set in $PG(n,q)$ contained in $B$.\end{proof}

\section{The linear code generated by the incidence matrix of points and $k$-spaces in $PG(n,q)$} \label{sect2}
In this section, we investigate the codewords of small weight in the $p$-ary linear code generated by the incidence matrix of points and $k$-dimensional spaces, or for short \emph{$k$-spaces}, in $PG(n,q)$, $q=p^h$, $p$ prime, $h\geq 1$. 

\begin{lemma}  \label{lem1}If $U_1$ and $U_2$ are subspaces of dimension at least $n-k$ in $PG(n,q)$, then $U_1-U_2\in C_k^\bot$.
\end{lemma}
\begin{proof} For every subspace $U_i$ of dimension at least $n-k$ and every $k$-space $K$, $(K,U_i)=1$, hence $(K,U_1-U_2)=0$, so $U_1-U_2\in C_k^\bot$.
\end{proof}
Note that in Lemma \ref{lem1}, $\dim U_1 \neq \dim U_2$ is allowed.
\begin{lemma} \label{lem2} There exists a constant $a\in \mathbb{F}_p$ such that $(c,U)=a$, for all subspaces $U$ of dimension at least $n-k$.

\end{lemma}
\begin{proof} Lemma \ref{lem1} yields $U_1-U_2\in C_k^\bot$, for all subspaces $U_1,U_2$ with $\dim(U_i)\geq n-k$, hence $(c,U_1-U_2)=0$, so $(c,U_1)=(c,U_2)$.
\end{proof}

\begin{theorem} \label{blset}The support of a codeword $c\in C_k$ with weight smaller than $2q^k$, for which $(c,S)\neq 0$ for some $(n-k)$-space $S$, is a minimal $k$-blocking set in $PG(n,q)$. Moreover, $c$ is a codeword taking only values from $\lbrace 0,a \rbrace$, $a\in \mathbb{F}_p^\ast$, and $supp(c)$ intersects every $(n-k)$-dimensional space in $1\pmod{p}$ points.
\end{theorem}
\begin{proof} If $c$ is a codeword with weight smaller than $2q^k$, and $(c,S)=a\neq 0$ for some $(n-k)$-space, then, according to Lemma \ref{lem2}, $(c,S)=a$ for all $(n-k)$-spaces $S$, so $supp(c)$ defines a $k$-blocking set $B$.

 Suppose that every $(n-k)$-space contains at least two points of the $k$-blocking set $B$. Counting the number of incident pairs $(P\in B, (n-k)$-space through $P$) yields $$\vert B \vert \left [\begin{array}{c} n\\n-k \end{array}\right]\geq\left [\begin{array}{c} n+1\\n-k+1 \end{array}\right]2.$$
 Using $\vert B \vert < 2q^k$ gives a contradiction. So there is a point $R\in B$ on a tangent $(n-k)$-space. Since $c_R$ is equal to $a$, according to Lemma \ref{lem2}, $c_{R'}=a$ for every essential point $R'$ of $B$.
 
 Suppose $B$ is not minimal, i.e. suppose there is a point ${R}\in B$ that is not essential. By induction on the dimension, we find an $(n-k-1)$-dimensional space $\pi$ tangent to $B$ in ${R}$.
 If every $(n-k)$-space through $\pi$ contains two extra points of $B$, then $\vert B \vert >2q^k$, a contradiction. Hence, there is an $(n-k)$-space $S$, containing besides ${R}$ only one extra point ${R}'$ of $supp(c)$,  such that $(c,S)=c_{{R}}+c_{{R}'}=a$. 
 But since $B$ is uniquely reducible to a minimal blocking set $B$ (see Theorem \ref{uni}), ${R}'$ is essential, hence, $c_{{R}'}=a$. But this implies that $c_{{R}}=0$, a contradiction.
 We conclude that  the $k$-blocking set $B$ is minimal.
 
 Since all the elements $R$ of $supp(c)$ have the coordinate value $c_R=a$, and since $(c,H)=a $ for every $(n-k)$-dimensional space $H$, necessarily $supp(c)$ intersects every $(n-k)$-dimensional space in $1\pmod{p}$ points.
\end{proof}

\begin{theorem} \label{modp} Let $c$ be a codeword of $C_k(n,q)$, $q=p^h$, $p>3$, with weight smaller than $2q^k$, for which $(c,S)\neq 0$ for some $(n-k)$-space $S$. Every subspace of $PG(n,q)$ that intersects $supp(c)$ in at least one point, intersects it in 1 $\pmod{p}$ points.
\end{theorem} 
\begin{proof}It follows from Theorem \ref{blset} that a codeword $c$ of $C_k(n,q)$ with weight smaller than $2q^k$, for which $(c,S)\neq 0$ for some $(n-k)$-space $S$, is a minimal $k$-blocking set $B$ of $PG(n,q)$, intersecting any $(n-k)$-space in $1\pmod{p}$ points. Using the same counting arguments as in the proof of Theorem \ref{19} (with $E=p$), shows that 
\[|B|(|B|-1) -(1+p)|B|\left(\frac{q^n-1}{q^{n-k}-1}\right)
+(1+p)\left(\frac{(q^{n+1}-1)(q^n-1)}{(q^{n-k+1}-1)(q^{n-k}-1)}\right)\geq 0.\] Substituting the values $\vert B\vert=2q^k-1$ and $\vert B\vert = 3(q^k+1)/2$ in this inequality yields a contradiction for $p>3$, hence $\vert B \vert<3(q^k+1)/2$. In \cite[Theorem 2.7]{TS:??}, it is proven that a subspace that intersects a minimal $k$-blocking set of size smaller than $3(q^k+1)/2$ in at least 1 point, intersects it in $1 \pmod{p}$ points.
\end{proof}

\textbf{We emphasize that from now on, for some of the results, it is necessary to  assume that $k\geq n/2$.}\\

The following lemmas are extensions of the lemmas in \cite{art}; we include the proofs to illustrate where the extra requirement $k\geq n/2$ arises.

\begin{lemma} \label{lem5} (See \cite[Lemma 3]{art}) Assume $k\geq n/2$. A codeword $c$ of $C_k$ is in $C_k\cap C_k^\bot$ if and only if $(c,U)=0$ for all subspaces $U$ with $\dim(U)\geq n-k$.
\end{lemma}
\begin{proof} Let $c$ be a codeword of $C_k\cap C_k^\bot$. Since $c\in C_k^\bot$, $(c,K)=0$ for all $k$-spaces $K$, Lemma \ref{lem2} yields that $(c,U)=0$ for all subspaces $U$ with dimension at least $n-k$ since $k\geq n-k$.
Now suppose $c\in C_k$ and $(c,U)=0$ for all subspaces $U$ with dimension at least $n-k$. Applying this to a $k$-space yields that $c\in C_k\cap C_k^\bot$ since $k\geq n-k$.
\end{proof}
\begin{remark} \label{opm2}If $k<n/2$, the lemma is false. Let $c$ be $K_1-K_2$, with $K_1$ and $K_2$ two skew $k$-spaces. It is clear that $c\in C_k$ and that $(c,S)=0$ for all $(n-k)$-spaces $S$. But $c\notin C_k^\bot$ since $(c,K_1)=1\neq0$.
Note that the lemma is still valid in one direction: if $c \in C_k\cap C_k^{\bot}$, then $(c,S)=0$ for all $(n-k)$-spaces. For, let $S$ be an $(n-k)$-space, and let $K_i$, $i=1,\ldots,\theta_{n-2k}$, be the $\theta_{n-2k}$ $k$-spaces through a fixed $(k-1)$-space $K'$ contained in $S$. Since $(c,K)=0$ for all $k$-spaces $K$, it follows that
$(c,S)=$ $(c,K_1\setminus K')+\cdots+(c,K_{\theta_{n-2k}}\setminus K')+(c,K')=0$.
\end{remark}

\begin{lemma} (See \cite[Lemma 4]{art}) \label{lem6} For $k\geq n/2$, 
$$C_k\cap C_k^\bot=\left\langle K_1-K_2 \vert\vert K_1, K_2 \mbox{ distinct $k$-spaces in }PG(n,q) \right \rangle.$$
\end{lemma}
\begin{proof}
Put $A=\left\lbrace K_1-K_2 \vert\vert K_1, K_2 \mbox{ distinct $k$-spaces in }PG(n,q) \right \rbrace$. Since $k\geq n/2$, two $k$-spaces $K$ and $K'$ of $PG(n,q)$ intersect in $1\pmod{p}$ points, so $(K,K')=1$. Hence, $A\subseteq C\cap C^\bot$, since $(K,v)=(K,K_i)-(K,K_j)=1-1=0$, for every $k$-space $K$ of $PG(n,q)$, and for every $v=K_i-K_j\in A$. 

Moreover, since $\left \langle A \cup \lbrace K_i \rbrace \right \rangle$ contains each $k$-space, it follows that $\dim(C)-1\leq \dim(\langle A \rangle)\leq$ $\dim(C\cap C^\bot)$. The lemma now follows easily, since $C\cap C^\bot$ is not equal to $C$, as a $k$-space, with $k\geq n/2$, %weer!
 is not orthogonal to itself.
\end{proof}

\begin{remark} If $k<n/2$, the lemma is false, since $K_1-K_2\notin C_k\cap C_k^\bot$, with $K_1$, $K_2$ two skew $k$-spaces (see Remark \ref{opm2}). \end{remark}

%\subsection{Excluding linear blocking sets}
The following lemmas are extensions of Lemmas 6.6.1 and 6.6.2 of Assmus and Key \cite{AK}. They will be used to exclude non-trivial small linear blocking sets as codewords. The proofs are an  extension of the proofs of Lemmas 7 and 8 of \cite{art}.

\begin{lemma} \label{lem7} For $k\geq n/2$, a  vector $v$ of $V(\theta_n,p)$ taking only values from $\lbrace 0,a \rbrace$, $a\in \mathbb{F}_p^\star$, is contained in $(C_k\cap C_k^\bot)^\bot$ if and only if $|supp(v) \cap K|\pmod{p}$ is independent of the $k$-space $K$ of $PG(n,q)$.
\end{lemma}

%\begin{proof}If $v \in$ $(C_k\cap C_k^\bot)^\bot$, then $(v, K_1-K_2)=0$ since $C_k\cap C_k^\bot$ is generated by the differences of the $k$-spaces (Lemma \ref{lem6}).
%We see that $(v,K)= a |supp(v) \cap K|$ mod $p$ is independent of the choice of the $k$-space $K$ and so is $|supp(v) \cap K|\pmod{p}$. 

%Conversely, if $|supp(v) \cap K|$ is constant $ $mod $p$, then $(v,K) = a \vert supp(v) \cap K\vert$ mod $p$. This implies that $(v, K_1-K_2)=0$ for all $k$-spaces $K_1,K_2$, and hence $v\in(C_k\cap C_k^\bot)^\bot$. 
%\end{proof}
\begin{remark} \label{opm3} If $k<n/2$, the lemma is false. Let $v$ be a $k$-space. It follows that $v\in (C_k\cap C_k^\bot)^\bot$ since $v\in C_k=(C_k^\bot)^\bot \subseteq (C_k \cap C_k^\bot)^\bot$. But $\vert supp(v)\cap K\vert$ is 0 $\pmod{p}$ or 1 $\pmod{p}$, depending on the $k$-space $K$.

\end{remark}
\begin{lemma} \label{lem8} Assume $k\geq n/2$ and let $c$,$v$ be two vectors taking only values from $\lbrace 0,a \rbrace$, for some $a \in \mathbb{F}_p^\star$, with $c$ $ \in C_k$, $v \in (C_k\cap C_k^\bot)^\bot$. If $\vert supp(c) \cap K\vert\equiv \vert supp(v) \cap K\vert\pmod{p}$ for every $k$-space $K$, then $\vert supp(c) \cap supp(v)\vert \equiv \vert supp(c)\vert\pmod{p}$.
\end{lemma}

%\begin{proof} We know that $c$ $\in$ $C$, so (see Lemma \ref{lem6}), $(c,K_1-K_2)=0$ for all $k$-spaces $K_1,K_2$, %k groter dan n/2!!$
%so $\vert supp(c) \cap K \vert$ is independent of the $k$-space $K$. Since $(c-v,K)=(c,K)-(v,K)\equiv a\vert supp(c) \cap K\vert-a\vert supp(v) \cap K\vert \equiv 0$ mod $p$, for every $k$-space $K$, we can derive that $c-v \in C^{\bot}$. So $(c-v,c)\equiv a^2\vert supp(c)\vert-a^2\vert supp(c) \cap supp(v)\vert\equiv 0$ mod $p$ since $c \in C$ en $c-v \in C^{\bot}$. This yields that $\vert supp(c)\vert \equiv \vert supp(c)\cap supp(v) \vert$ mod $p$. \end{proof}

As mentioned in the introduction, we will eliminate all so-called non-trivial {\it linear} $k$-blocking sets as the support of a codeword of $C$ of small weight. In order to define a linear $k$-blocking set, we introduce the notion of a Desarguesian spread.

By what is sometimes called "field reduction", the points of $PG(n,q)$, $q=p^h$, $p$ prime, $h\geq 1$, correspond to $(h-1)$-dimensional subspaces of $PG((n+1)h-1,p)$, since a point of $PG(n,q)$ is a $1$-dimensional vector space over ${\mathbb F}_q$, and so an $h$-dimensional vector space over ${\mathbb F}_p$. In this way, we obtain a partition ${\mathcal D}$ of the point set of $PG((n+1)h-1,p)$ by $(h-1)$-dimensional subspaces. In general, a partition of the point set of a projective space by subspaces of a given dimension $k$ is called a {\it spread}, or a {\it $k$-spread} if we want to specify the dimension. The spread we have obtained here is called a {\it Desarguesian spread}. Note that the Desarguesian spread satisfies the property that each subspace spanned by two spread elements is again partitioned by spread elements. In fact, it can be shown   that 
%if the dimension of the ambient space is larger than twice the dimension plus one (i.e. 
if $n\geq 2$, this property characterises a Desarguesian spread \cite{L1}.

%From now on, we use the representation of minimal blocking sets in terms of spreads (see Lunardon \cite{L1},\cite{L2}). 
\begin{definition} Let $U$ be a subset of $PG((n+1)h-1,p)$ and let $\mathcal{D}$ be a Desarguesian $(h-1)$-spread of $PG((n+1)h-1,p)$, then $\mathcal{B}(U)=\lbrace R \in \mathcal{D}||U\cap R \neq \emptyset \rbrace$.\end{definition}

Analogously to the correspondence between the points of $PG(n,q)$ and the elements of a Desarguesian spread $\mathcal D$ in $PG((n+1)h-1,p)$, we obtain the correspondence between the lines of $PG(n,q)$ and the $(2h-1)$-dimensional subspaces of $PG((n+1)h-1,p)$ spanned by two elements of $\mathcal D$, and in general, we obtain the correspondence between the $(n-k)$-spaces of $PG(n,q)$ and the $((n-k+1)h-1)$-dimensional subspaces of $PG((n+1)h-1,p)$ spanned by $n-k+1$ elements of $\mathcal D$. With this in mind, it is clear that any $hk$-dimensional subspace $U$ of $PG(h(n+1)-1,p)$ defines a $k$-blocking set ${\mathcal B}(U)$  in $PG(n,q)$. A blocking set constructed in this way is called a {\it linear $k$-blocking set}.  Linear $k$-blocking sets were first introduced by Lunardon \cite{L1}, although there a different approach is used.
For more on the approach explained here, we refer to \cite{lavrauw2001}.

The following lemmas, theorems, and remarks are proven in the same way as the authors do in \cite{art}.

\begin{lemma}\cite[Lemma 9]{art}\label{lem9}
If  $U$ is a subspace of $PG((n+1)h-1,q)$, then $\vert \mathcal{B}(U) \vert \equiv 1\ mod\ q$.
\end{lemma}
%\begin{proof}
%Suppose $U$ is a subspace of $PG((n+1)h-1,q)$ of dimension $r$ and let $X_i$ be the number of spreadelements intersecting $U$ in a subspace of dimension $i$. Each point of $U$ lies in a unique spreadelement, so

%$$\sum_{i=0}^r X_i\theta_i=\theta_r \Leftrightarrow$$
%$$\sum_{i=0}^r X_iq^{i+1}-\sum_{i=0}^r X_i=q^{r+1}-1 \Leftrightarrow$$
%$$q(\sum_{i=0}^rX_iq^i-q^r)=\sum_{i=0}^rX_i-1.$$

%The left hand side is divisible by $q$, so $\sum_{i=0}^r X_i=\vert \mathcal{B}(U) \vert =1$ mod $p$.
%\end{proof}

We put $N=hk$ throughout the following results. We call a linear $k$-blocking set $B$ of $PG(n,q)$, $q=p^h$, $p$ prime, 
$h\geq 1$,  defined by an $N$-dimensional space of $PG(h(n+1)-1,p)$ a {\em small} linear $k$-blocking set.
%NIEUW
\begin{lemma}\cite[Lemma 10]{art}\label{N1} Let $U_N$ be an $N$-dimensional subspace of $PG(h(n+1)-1,p)$. The number of spread elements of $\mathcal{B}(U_N)$ intersecting $U_N$ in exactly one point is at least $p^{N}-p^{N-2}-p^{N-3}-\cdots-p^{N-h+1}-p^{N-h-2}-\cdots-p^{N-2h+1}-p^{N-2h-2}-\cdots-p^{h+1}-p^{h-2}-\cdots-p$.\end{lemma}
%\begin{proof} Suppose there are exactly $x$ spreadelements of $\mathcal{B}(U_N)$ intersecting $U_N$ in one point, then
%$$\frac{p^{h(k+1)}-1}{p^h-1}\leq \vert B \vert \leq \frac{\vert U_N\vert-x}{p+1}+x.$$
%Using that $\vert U_N \vert=(p^{N+1}-1)/(p-1)$ yields that $x\geq p^{N}-p^{N-2}-p^{N-3}-\cdots-p^{N-h+1}-p^{N-h-2}-\cdots-p^{N-2h+1}-p^{N-2h-2}-\cdots-p^{h+1}-p^{h-2}-\cdots-p$.
%\end{proof}

\begin{remark} \label{rem1} It follows from Lemma \ref{N1} that the number of spread elements of $\mathcal{B}(U_N)$ intersecting $U_N$ in exactly one point is at least $p^{N}-p^{N-1}+1$. We will use this weaker bound.
\end{remark}

\begin{lemma}\cite[Lemma 11]{art} \label{N2} If there are $p^{N}-p^{N-1}+1$ points $R_i$ of a minimal $k$-blocking set $B$ in $PG(n,q)$, for which it holds that every line through $R_i$ is either a tangent line to $B$ or is entirely contained in $B$, then $B$ is a $k$-space of $PG(n,q)$.
\end{lemma}
%\begin{proof}
%It is easy to see that a plane through a line $R_iR_j$, $i\neq j$ is either completely contained in $B$, or intersects $B$ only in the line $R_iR_j$.
%There are at least $(p^{N}-p^{N-1})/p^h$ different lines $R_1R_i$, $i\neq 1$.

%We prove that if $B \supset \pi_m$, $B \neq \pi_m$, for some $m$-space $\pi_m$ through $R_1$, then $B \supseteq \pi_{m+1}$ for some $(m+1)$-space $\pi_{m+1}$ through $\pi_m$ for all $m< k$. 

%If $B\supset \pi_m$, then there are still $(p^{N}-p^{N-1})/p^h-(p^{hm}-1)/(p^h-1)$ lines $R_1R_j$ through $R_1$ such that every plane through it intersects $B$ in this line or lies completely in $B$. We can choose such a line $R_1R_j$ if $m<k$ and $p>2$. Then the space $\lbrace R_1R_j,\pi_m\rbrace$ is clearly contained in $B$.
%By induction, we find a $k$-space $\pi$ contained in $B$. Since $B$ is minimal, $B=\pi$.
%\end{proof}

\begin{remark} It follows from the proof of Lemma 11 in \cite{art} that it is sufficient to find $k$ linearly independent points $R_i$ such that every line through $R_i$ is either a tangent line to $B$ or is entirely contained in $B$ to prove that $B$ is a $k$-space. Moreover, this bound is tight. If there are only $k-1$ linearly independent points for which this condition holds, we have the counterexample of a Baer cone, i.e. let $B$ be the set of all lines connecting a point of a Baer subplane $\pi=PG(2,\sqrt{q})$ to the points of a $(k-2)$-dimensional subspace of $PG(n,q)$, skew to $\pi$.
\end{remark}

%
%Einde nieuw

\begin{lemma}\cite[Lemma 12]{art}\label{lem10b}
Let $U_{N-1}$ be a fixed $(N-1)$-space in $PG(h(n+1)-1,p)$ and let $U_N$ be an arbitrary $N$-space containing $U_{N-1}$. The set $\mathcal{B}(U_N)$ is entirely determined by $U_{N-1}$ and two elements $R_1$, $R_2$ $\in$ $\mathcal{B}(U_N) \backslash \mathcal{B}(U_{N-1})$.
\end{lemma}

%\begin{proof} We may assume that $\mathcal{B}(U_{N-1}) \neq \mathcal{B}(U_N)$, since the theorem is obvious if $\mathcal{B}(U_{N-1}) = \mathcal{B}(U_N)$. Suppose $R_1,R_2 \in \mathcal{B}(U_N)\backslash \mathcal{B}(U_{N-1}), R_1 \neq R_2$. If $R_3 \in \mathcal{B}(U_N) \backslash \mathcal{B}(U_{N-1})$, $R_2 \neq R_3 \neq R_1$then we claim that $R_3$ can be constructed only using elements of $\mathcal{B}(U_{N-1}) \cup \lbrace R_1,R_2\rbrace$.
%Clearly, $R_i$ intersects $U_N$ in a point $p_i$ since $R_1, R_2$ and $R_3$ are elements of $\mathcal{B}(U_N) \backslash \mathcal{B}(U_{N-1})$. So $\langle p_1, p_3 \rangle$ intersects $U_{N-1}$ in a point $p_4$ which lies on a unique spreadelement $R_4$. Similarly, the spreadelement through $\langle R_2,R_3 \rangle \cap U_{N-1}$ is called $R_5$.

%Case 1: $p_3 \notin p_1p_2$.
%The intersection of $\langle R_1,R_4 \rangle$ with $\langle R_2,R_5\rangle$ certainly contains $R_3$. $\langle R_1,R_4 \rangle$ and $\langle R_2,R_5\rangle$ are spaces spanned by two elements of a Desarguesian spread, so they intersect in a spreadelement. We can conclude that $R_3=\langle R_1,R_4\rangle \cap \langle R_2,R_5 \rangle$.

%Case 2: $p_3 \in p_1p_2$.
%Take a spreadelement $R_6 \in U_N$ already constructed in case 1. We can switch $R_6$ with $R_2$. Then $R_3 \notin \langle R_1,R_6 \rangle$. So we can copy the proof of case 1 to determine $R_3$ from $R_1$,$R_6$ and $U_{N-1}$. But $R_6$ was determined by $R_1,R_2$ and $U_{N-1}$, hence so is $R_3$.
%\end{proof}

\begin{theorem}\label{th10}
For every small  linear $k$-blocking set $B$, not defining a $k$-space in $PG(n,p^h)$, there exists a small  linear $k$-blocking set $B'$ intersecting $B$ in $2\pmod{p}$ points.
\end{theorem}
\begin{proof}
As we have seen before, a linear $k$-blocking set $B$ in $PG(n,p^h)$ corresponds  to an $N$-space $U_N$ in $PG(h(n+1)-1,p)$. We will construct a subspace $U_{N}'$ that defines a second $k$-blocking set $B'$ intersecting $B$ in $2 \pmod{p}$ points.

Choose a spread element $R_1$ intersecting $U_N$ in 1 point, say $p_1$. The element $R_1$ exists because of Lemma \ref{N1}.
%: when every spread element intersects $U_N$ in a line, there would be $\vert U_N\vert /(p+1)$ spread elements intersecting $U_N$ giving rise to less than $\theta_{k}$ points in $PG(n,p^h)$, which is the number of points in the smallest blocking set i.e. a $k$-space.
Choose an $(N-1)$-dimensional subspace $U_{N-1}\subseteq U_N$ not intersecting $R_1$.

We can choose a spread element $R_2 \in \mathcal{B}(U_{N-1})$ not lying in $U_{N-1}$. Suppose that all elements of $\mathcal{B}(U_{N-1})$ lie in $U_{N-1}$. Then there are $\theta_{k-1}$ elements in $\mathcal{B}(U_{N-1})$. Every element of $\mathcal{B}(U_N)\backslash \mathcal{B}(U_{N-1})$ has to intersect $U_N$ in a point, so there are $p^{hk}$ elements in $\mathcal{B}(U_N)\backslash \mathcal{B}(U_{N-1})$. Hence, there are in total $\theta_{k}$ spread elements in $\mathcal{B}(U_N)$, corresponding to a $k$-space in $PG(n,p^h)$, since this is the only $k$-blocking set in $PG(n,p^h)$ of size $\theta_k$, a contradiction. So there is a spread element $R_2\in \mathcal{B}(U_{N-1})$ not contained in $U_{N-1}$.

Suppose that for every $R'_1$ with $\vert R'_1\cap U_N\vert=1$, and $R'_2$ in $\mathcal{B}(U_{N-1})$, each spread element in $\langle R'_1,R'_2\rangle$ intersects $U_N$. Then $\mathcal{B}(U_N)$ defines a set of points in $PG(n,q)$ such that every line through $R'_1$ is tangent to $B$ in $R'_1$ or is entirely contained in $B$. But Remark \ref{rem1} and Lemma \ref{N2} imply that $B$ is a $k$-space, a contradiction.
So there is a spread element $R'$, lying in a $(2h-1)$-space spanned by two spread elements $R_1$ and $R_2$, $R_1\in \mathcal{B}(U_N)$, where $R_1\cap U_N$ is a point, and $R_2\in \mathcal{B}(U_{N-1})$, such that $R'$ does not intersect $U_N$. 

The elements $R_1,R_2,R'$ define an $(h-1)$-regulus. Take the transversal line $m$ intersecting $U_{N-1}$ in a point of $U_{N-1} \cap R_2$. Then $\langle m,U_{N-1} \rangle$ is an $N$-space $U_N'$, defining a $k$-blocking set $B'$ in $PG(n,p^h)$.

Now $\mathcal{B}(U_N)$ and $\mathcal{B}(U_N')$ have $\mathcal{B}(U_{N-1})$ and $R_1$ in common. So $B$ and $B'$ have at least $(1 \mod p) +1$ points in common (see Lemma \ref{lem9}). 

If $\mathcal{B}(U_N)\cap \mathcal{B}(U'_N)$ contains another spread element $R_3 \notin \mathcal{B}(U_{N-1})$, $R_3 \neq R_1$, then Lemma \ref{lem10b} implies that $\mathcal{B}(U_N)=\mathcal{B}(U'_N)$, contradicting $R'\in \mathcal{B}(U'_N)\backslash \mathcal{B}(U_N)$. It follows that the $k$-blocking sets $B$ and $B'$ corresponding to $U_N$ and $U'_N$, resp., intersect in $2 \pmod{p}$ points. 
\end{proof}
Using this, we exclude in Theorem \ref{th4} all small non-trivial linear $k$-blocking sets as codewords.
\begin{theorem}\label{th4}
Assume $k\geq n/2$. If $v$ is the incidence vector of a small non-trivial linear $k$-blocking set in $PG(n,q)$, then $v\notin C_k(n,q)$. 
\end{theorem}
\begin{proof} Let $q=p^h$, $p$ prime, $h\geq 1$.
We know that $\vert supp(v) \vert\equiv1\pmod{p}$, since $supp(v)$ corresponds to $\mathcal{B}(U)$ for some subspace $U$ in $PG((n+1)h-1,p)$, and $\vert \mathcal{B}(U)\vert\equiv1 \pmod{p}$ (see Lemma \ref{lem9}). We know from Theorem \ref{th10} that there exists a small linear $k$-blocking set $w$ such that $\vert supp(v) \cap supp(w) \vert\equiv 2 \pmod{p}$. Since $\vert supp(w) \cap K\vert \equiv 1\pmod{p}$ for every $k$-space $K$ (Lemma \ref{lem9}), it follows that $w \in (C_k \cap C_k^\bot)^\bot$ (Lemma \ref{lem7}). Similarly, $\vert supp(v) \cap K\vert \equiv 1\pmod{p}$, for every $k$-space $K$. Suppose that $v \in C_k$. Lemma \ref{lem8} implies that $\vert supp(v) \cap supp(w)\vert \equiv \vert supp(v) \vert\pmod{p}$ $\equiv 1\pmod{p}$, a contradiction.
\end{proof}

\begin{corollary} \label{cor2} For  $k\geq n/2$, the only possible codewords $c$ of $C_k(n,q)$ of weight in $]\theta_k,2q^k[$, such that $(c,S)\neq 0$ for an $(n-k)$-space $S$, are scalar multiples of non-linear minimal $k$-blocking sets of $PG(n,q)$.
\end{corollary}

\begin{remark}
In view of Corollary \ref{cor2} it is important to mention the conjectures made in \cite{sziklai}. If these conjectures are true (i.e. all small minimal blocking sets are linear), then Corollary \ref{cor2} eliminates all codewords of $C_k(n,q)\backslash C_k(n,q)^\bot$ of weight in the interval $]\theta_k,2q^k[$.
\end{remark}

For $q=p$ prime and for $q=p^2$, $p>11$ prime, we can exclude all such possible codewords. We rely on  the following results.

\begin{theorem}\label{th:1}
The only minimal $k$-blocking sets $B$ in $PG(n,p)$, with $p$ prime and $|B|<2p^k$, such that every $(n-k)$-space intersects $B$ in $1\pmod{p}$ points, are $k$-spaces of $PG(n,p)$.
\end{theorem}

\begin{proof}
By induction on the dimension, it is possible to prove that if a line contains at least two points of $B$, then this line is contained in $B$. It now follows, by induction on the dimension,  that $B$ is a $k$-space.
\end{proof}
To exclude codewords in $C_k(n,p^2)$, with $p$ a prime, we can use the following theorem of Weiner which implies that every small minimal blocking set in $PG(n,p^2)$ is linear.
\begin{theorem}\label{Weiner} \cite{weiner} A non-trivial minimal $(n-k)$-blocking set of $PG(n,p^2)$, $p>11$, $p$ prime, of size less than $3(p^{2(n-k)}+1)/2$ is a $(t,2((n-k)-t-1))$-Baer cone with as vertex a $t$-space and as base a $2((n-k)-t-1)$-dimensional Baer subgeometry, where $\max \lbrace -1,n-2k-1 \rbrace \leq t < n-k-1$.
\end{theorem}

Theorems \ref{th:1} and  \ref{Weiner}, together with Corollary \ref{cor2}, yield the following corollary.
\begin{corollary} \label{kw}\label{th6}
There are no codewords $c$, with $wt(c) \in ]\theta_k,2q^k[$, in $C_k(n,q)\backslash $ $C_k(n,q)^\perp$, with $k\geq n/2$, $q$ prime or $q=p^2$, $p>11$, $p$ prime.
\end{corollary}

\section{The dual code of $C_k(n, q)$}
In this section, we consider codewords $c$ in the dual code $C_k(n,q)^\bot$ of $C_k(n,q)$. 
The goal of this section is to find a lower bound on the minimum weight of the code $C_k(n,q)^\bot$. Denote the minimum weight of a code $C$ by $d(C)$.

In the following lemmas, the problem of finding the minimum weight of $C_k(n,q)^\bot$ is reduced to finding the minimum weight of $C_1(n-k+1,q)^\bot$. Note that $d(C_k(n,q)^\perp)\leq 2q^{n-k}$ since the difference of the incidence vectors of two $(n-k)$-spaces of $PG(n,q)$, intersecting in an $(n-k-1)$-space, is a codeword of $C_k(n,q)^\perp$.

\begin{lemma}\label{min}
For each $n\geq 2$, $0<k\leq n-1$, the following inequalities hold:
$$d(C_k(n,q)^\bot) \geq d(C_{k-1}(n-1,q)^\bot)\geq \cdots \geq d(C_1(n-k+1,q)^\bot).$$
\end{lemma}
\begin{proof}
Let $c$ be a codeword of $C_k(n,q)^\bot$ of minimum weight, let $R$ be a point of $PG(n,q)\backslash supp(c)$, lying in a tangent line to $supp(c)$,  and let $H$ be a hyperplane of $PG(n,q)$ not containing $R$. For each point $P\in H$, define $c'_P=\sum c_{P_i}$, with $P_i$ the points of $supp(c)$ on the line $\langle R,P\rangle$, and let $c'$ denote the vector with coordinates $c'_P$, $P\in H $. 
It easily follows that $c'\in C_{k-1}(n-1,q)^\bot$, and $supp(c')$ is contained in the projection of $supp(c)$ from the point $R$ onto the hyperplane $H$. Clearly, $\vert supp(c')\vert \leq \vert supp(c) \vert$. 
Using this relation on a codeword $c$ of minimum weight  yields that 
$d(C_{k-1}(n-1,q)^\bot)\leq d(C_k(n,q)^\bot)$. Continuing this process proves the statement.
\end{proof}

%\begin{remark} We call the vector $c'$ defined in the proof of Lemma \ref{min}, the {\em projection} of $c$.
%\end{remark}
\begin{theorem}\label{st1} For each $n\geq2$, $0<k\leq n-1$, $d(C_k(n,q)^\bot)=d(C_1(n-k+1,q)^\bot)$.
\end{theorem}
\begin{proof}
Embed $\pi=PG(n-k+1,q)$ in $PG(n,q)$, $n>2$, and extend each codeword $c$ of $C_1(\pi)^\bot$ to a vector $c^{(n)}$ of $V(\theta_n,p)$ by putting a zero at each point $P\in PG(n,q)\backslash \pi$. Since the all one vector of $V(\theta_{n-k+1},p)$ is a codeword of $C_1(n-k+1,q)$, it follows that $\sum_{P\in \pi} c_P^{(n)}=0$ for each $c^{(n)}$. 
This implies that $(c^{(n)},K)=0$, for each $k$-space $K$ of $PG(n,q)$ which contains $\pi$. If a $k$-space $K$ of $PG(n,q)$ does not contain $\pi$, then $(c^{(n)},K \cap \pi)=0$, since $K\cap \pi$ is a line or can be described as a pencil of lines through a given point, and $(c,l)=0$ for each line $l$ of $\pi$.
 It follows that $c^{(n)}$ is a codeword of $C_k(n,q)^\bot$ of weight equal to the weight of $c$, which implies that $d(C_k(n,q)^\bot)\leq d(C_1(n-k+1,q)^\bot)$. Regarding Lemma \ref{min}, this yields that $d(C_k(n,q)^\bot)=d(C_1(n-k+1,q)^\bot)$.
\end{proof}

\begin{lemma} \label{l1}
Let $B$ be a set of points in $PG(n,q)$, with the property that those points of $PG(n,q)\setminus B$ that are incident with a secant line to $B$ are incident with no tangent lines to $B$. If $\dim\langle B\rangle \geq n-k+2$, then $\vert B \vert \geq \theta_{n-k+1}$.
\end{lemma}

\begin{proof}
We first prove the following result. 

Let $P$ be a point in $B$ and let $L$ be a line through $P$, lying in a plane $\pi$ through $P,R,S$, with $R,S \in B$ and $P\notin RS$, then $L$ is a secant line to $B$. If $L$ is a tangent line to $B$, then the point $RS\cap L$ lies on a secant line and on a tangent line, a contradiction.

By induction, we prove that for each point $P\in B$, there exists an $r$-space $\pi_r$, with $r\leq n-k+2$, such that all lines through $P$ in $\pi_r$ are secant lines. The case $r=2$ is already settled, so suppose that the statement is true for $r$, $r< n-k+2$. There is a point $T\in B\notin \pi_r$ since $\dim\langle B \rangle \geq n-k+2$. If $M$ is a line through $P$ in $\langle \pi_r,T\rangle$, then $\langle M, T\rangle$ intersects $\pi_r$ in a line $N$ through $P$, which is a secant line according to the induction hypothesis. Hence, we find three non-collinear points in $B$ in the plane $\langle N,T\rangle$, so $M$ is a secant line, so there is an $(r+1)$-space for which any line through $P$ is a secant line. Counting the points of $B$ on  lines through $P$ yields that $\vert B \vert \geq \theta_{n-k+1}$.
%We first prove the following result. When we take two secants $l_1, l_2$ through $R$, then the plane $\langle l_1,l_2 \rangle$ contains at least $q+ \max \lbrace a_1,a_2 \rbrace$ points of $B$, where $a_i=\vert l_i \cap B \vert$. 
%Take a point $S\in B$ on $l_1\backslash l_2$. Then every line in $\langle l_1,l_2 \rangle$ through $S$ must be a secant to $B$; else if it lies on a tangent line $l$, $l\cap l_2$ is a point not in $B$ lying on a tangent line and a secant to $B$, which is a contradiction. So $\vert B \cap \langle l_1,l_2 \rangle \vert \geq q+a_1$ and similarly, $\vert B \cap \langle l_1,l_2 \rangle\vert \geq q+a_2$.
%Now $R$ lies on at least three non-coplanar secants to $B$ since $\dim \langle B \rangle \geq 3$. Now $$\vert \langle l_1,l_2 \rangle \cap B \vert \geq q+\max\lbrace a_1,a_2 \rbrace$$
%$$\vert \langle l_1,l_3 \rangle \cap B \vert \geq q+\max\lbrace a_1,a_3 \rbrace$$
%$$\vert \langle l_2,l_3 \rangle \cap B \vert \geq q+\max\lbrace a_2,a_3 \rbrace.$$
%So $\vert B \vert \geq (q+\max\lbrace a_1,a_2 \rbrace)+ (q+\max\lbrace a_1,a_3 \rbrace)+(q+\max\lbrace a_2, a_3 \rbrace)$ $-(a_1+a_2+a_3)$, because we counted the points lying on $l_i\cap B$ twice. It follows that $\vert B \vert \geq 3q$.
\end{proof}
\begin{theorem} \label{th3}
If $c$ is a codeword of $C_k(n,q)^\bot$, $n\geq 3$, of minimal weight, then $supp(c)$ is contained in an $(n-k+1)$-space of $PG(n,q)$.
\end{theorem}
\begin{proof} As already observed, we may assume that $wt(c) \leq 2q^{n-k}$.
 Assume that $\dim\langle supp(c) \rangle \geq n-k+2$. Using Lemma \ref{l1}, we find a point $R\notin supp(c)$ lying on a tangent line to $supp(c)$ and lying on at least one secant line to $supp(c)$. It follows from Theorem \ref{st1} that $$wt(c)=d(C_k(n,q)^\bot)=d(C_{k-1}(n-1,q)^\bot)=d(C_1(n-k+1,q)^\bot).$$
Let $c'$ be defined as in the proof of Lemma \ref{min}.
Since $R$ lies on at least one secant line to $supp(c)$, $0<wt(c')<wt(c)$.
But this implies that $c'$ is a codeword of $C_{k-1}(n-1,q)^\bot$ satisfying $0<wt(c')\leq wt(c)-1<d(C_{k-1}(n-1,q)^\bot)$, a contradiction.
\end{proof}
In Theorem \ref{th3}, we proved that finding the minimum weight of the code $C_k(n,q)^\bot$ is equivalent to finding the minimum weight of the code $C_1(n-k+1,q)^\perp$ of points and lines in $PG(n-k+1,q)$. Hence, we can use the following result due to Bagchi and Inamdar.

\begin{result}\cite[Proposition 2]{Bagchi} \label{b1} When $q$ is prime, the minimum weight of the dual code $C_1(n,q)^\bot$ is $2q^{n-1}$. Moreover, the codewords of minimum weight are precisely the scalar multiples of the difference of two hyperplanes.
\end{result}

Using Result \ref{b1}, together with Theorem \ref{th3}, yields the following theorem.
\begin{theorem} \label{priem}The minimum weight of $C_k(n,p)^\bot$, where $p$ is a prime, is equal to $2p^{n-k}$, and the codewords of weight $2p^{n-k}$ are  the scalar multiples of the difference of two $(n-k)$-spaces intersecting in an $(n-k-1)$-space.
\end{theorem}
When $q$ is not a prime, this result is false; we will present some counterexamples.

\begin{theorem} \label{st}
Let $B$ be a minimal $(n-k)$-blocking set in $PG(n,q)$ of size $q^{n-k}+x$, with $x<(q^{n-k}+1)/2$, such that there exists an $(n-k)$-space $T$ intersecting $B$ in $x$ points. The difference of the incidence vectors of $B$ and $T$ is a codeword of $C_k(n,q)^\bot$ with weight $2q^{n-k}+\theta_{n-k-1}-x$.
\end{theorem}
\begin{proof}
If $x<(q^{n-k}+1)/2$, then $B$ is a small minimal $(n-k)$-blocking set, hence every $k$-space intersects $B$ in $1\pmod{p}$ points (see \cite{TS:??}). Let $c_1$ be the incidence vector of $B$ and let $c_2$ be the incidence vector of an $(n-k)$-space intersecting $B$ in $x$ points. Then $(c_1-c_2,K)=(c_1,K)-(c_2,K)=0$ for all $k$-spaces $K$, hence $c_1-c_2$ is a codeword of $C_k(n,q)^\bot$, with weight $\vert B\vert+\vert T \vert-2\vert B\cap T \vert =2q^{n-k}+\theta_{n-k-1}-x$.
\end{proof}
We can use this theorem to lower the upper bound on the possible minimum weight of codewords of $C_k(n,q)^\bot$.
Put $V(n+1,q)=V(1,q)\times V(n-k,q) \times V(k,q)=\mathbb{F}_q\times \mathbb{F}_{q^{n-k}}\times \mathbb{F}_{q^k}$ and put
$$B=\left\lbrace (1,x,Tr(x))\vert \vert x \in \mathbb{F}_{q^{n-k}} \right \rbrace \cup \left\lbrace (0,x,Tr(x))\vert \vert x \in \mathbb{F}_{q^{n-k}},x\neq 0 \right \rbrace,$$
where $Tr$ is the trace function of $\mathbb{F}_{q^{n-k}}$ to $\mathbb{F}_p$, $p$ prime. The set $B$ is a subset of $\mathbb{F}_q\times \mathbb{F}_{q^{n-k}}\times \mathbb{F}_{q}$ since $Tr(x)\in \mathbb{F}_p\subset \mathbb{F}_q, \forall x$. Moreover, $B$ is a linear subspace, inducing
a blocking set of size $q^{n-k}+(q^{n-k}-1)/(p-1)$, say $q^{n-k}+x$, w.r.t. the lines in $PG(\mathbb{F}_q\times \mathbb{F}_{q^{n-k}}\times \mathbb{F}_{q})\cong PG(n-k+1,q)$. Furthermore, there is an $(n-k)$-space $\pi$ such that $\vert B\cap \pi\vert=x$. Embedding $B$ in $PG(n,q)$ yields that $B$ is a minimal blocking set w.r.t. $k$-spaces, hence $B$ is a minimal $(n-k)$-blocking set such that there exists an $(n-k)$-space that intersects $B$ in $x$ points.

Using this, together with Theorem \ref{st}, yields the following corollary.
\begin{corollary}
For $q=p^h$, $p$ prime, $h\geq 1$, 
$$d(C_k(n,q)^\perp) \leq 2q^{n-k}+\theta_{n-k-1}-\frac{q^{n-k}-1}{p-1}.$$
\end{corollary}

 In the case where $q$ is even, $\cite{Bagchi}$ gives an upper bound on the minimum weight.
\begin{result}\label{ba}\cite[Proposition 4]{Bagchi} For $q$ even, the minimum weight of the code $C_1(n,q)^{\perp}$ is at most $q^{n-2}(q+2)$.
\end{result}
Result \ref{ba}, together with Theorem \ref{th3}, has the following corollary.
\begin{corollary}For $q$ even, the minimum weight of $C_k(n,q)^{\bot}$ is at most $q^{n-k-1}(q+2)$.
\end{corollary}

\begin{remark} It is easy to see that the minimum weight of $C_1(n-k+1,q)^\bot$, hence of $C_k(n,q)^\bot$, is at least $\theta_{n-k}+1$ since in $C_1(n-k+1,q)^\bot$, every line through a point of $supp(c)$, with $c\in C_1(n-k+1,q)^\bot$, has to contain at least one other point of $supp(c)$. If $q$ is odd, Theorems \ref{th16} and \ref{th8} improve this lower bound. If $q$ is even, then $d(C_k(n,q)^\bot)>\theta_{n-k}+1$, for $n>3$, since otherwise, $supp(c)$ would be a set $B$ of points in $PG(n-k+1,q)$, no three collinear, and \cite[Theorem 27.4.6]{Hirsch} states that $\vert B\vert \leq q^{n-k}-q^{n-k-1}/2+4q^{n-7/2}$, a contradiction. For $n=3$ and $k=1$, \cite[Lemma 16.1.4]{hir2} yields that $\vert B \vert \leq q^2+1$, a contradiction. For $n=3$ and $k=2$, it is easy to see that the minimum weight is $q+2$.

\end{remark}

We will now prove a lower bound on the minimum weight of $C_k(n,q)^\perp$, $q$ not a prime, $q$ odd, by extending the bound of Sachar \cite{sachar} on the minimum weight of $C_1(2,q)^\bot$.
%\begin{lemma}\label{lem30} If $c$ is a codeword of minimum weight in $C_k^\bot$, then $\theta_{n-k}+1\leq wt(c)\leq 2q^{n-k}$.
%\end{lemma}
%\begin{proof}
%If $S_1, S_2$ are two $(n-k)$-spaces intersecting in an $(n-k-1)$-space, then, according to Lemma \ref{lem1}, $S_1-S_2$ is a codeword of $C^\bot$, and it has weight $2q^{n-k}$. This proves the upper bound. We prove the lower bound by induction.

% Let $c$ be a codeword of $C_k(n,q)^\perp$. If $k=1$, then $supp(c)$ is a set of points such that every line of $PG(n,q)$ contains zero or at least two points of $supp(c)$. Looking at all lines through a point of $supp(c)$ yields that $wt(c)\geq \theta_{n-1}+1$. 
% 
% So suppose that $k>1$. It follows that $c$ has weight at most $2q^{n-k}$. There is a $(k-1)$-space $K'$ for which $(c,K')\neq 0$ since otherwise, $c$ would be a codeword of the dual of the code of points and $(k-1)$-spaces, which has weight at least $\theta_{n-k+1}+1$, a contradiction.
%Every $k$-space $K$ through $K'$ has to contain an extra point of $supp(c)$ in order to have $(c,K)=0$, hence $wt(c)\geq \theta_{n-k}+1$.
%\end{proof}

\begin{lemma}\label{lem12} Suppose that there are $2m$ different non-zero symbols used in the codeword $c\in C_k(n,q)^\bot$, $q$ odd. Then $$wt(c)\geq \frac{4m}{2m+1}\theta_{n-k}+\frac{2m}{2m+1}.$$

\end{lemma}
\begin{proof} We use the same techniques as in the proof of Proposition 2.2 in \cite{sachar}. Let $c$ be a codeword in $C_k^\bot$. Assume that $wt(c)\leq2q^{n-k}$, and write $wt(c) $ as $\theta_{n-k}+x$. 

Through every point $P$ of $supp(c)$, we can construct by induction on $s$, an $s$-space that only intersects $supp(c)$ in $P$, through a fixed $(s-1)$-space only intersecting $supp(c)$ in $P$, if $s\leq k-1$, since the number of $s$-spaces through an $(s-1)$-space is $(q^{n-s+1}-1)/(q-1)>2q^{n-k}$ if $n-s>n-k$.
So through every point $P$ of $supp(c)$, there is a $(k-1)$-space $K'$ which intersects $supp(c)$ only in the point $P$.
For simplicity of notations, we use the terminology {\em $2$-secant} for a $k$-space having two points of $supp(c)$.  Let $\bar{K}$ be a $(k-1)$-space intersecting $supp(c)$ in one point, for which the number of $2$-secants through $\bar{K}$ is minimal. We denote this number by $X$, or  by $X_R$ in case $\bar{K}$ intersects $supp(c)$ in the point $R$ of $supp(c)$.

 Since $c$ is orthogonal to every $k$-space, if $K$ is a $2$-secant through $R$ and $R'$, $R$, $R'\in supp(c)$, then $c_R+c_{R'}=0$, so the symbol $c_{R'}$ occurs at least $X$ times in $c$. In fact, the number of occurrences of a certain non-zero symbol is always at least $X$.

 The number of $2$-secants through a  given $(k-1)$-space intersecting $supp(c)$ in exactly one point, is at least $\theta_{n-k}-x+1$. So it is easy to see that the number of non-zero symbols used in $c$ must be even; let this number of non-zero symbols be $2m$.

 This implies that
 $$2m(\theta_{n-k}-x+1)\leq \theta_{n-k}+x.$$
 
 Hence, $$x\geq \frac{2m-1}{2m+1}\theta_{n-k}+\frac{2m}{2m+1},$$ and
 $$ wt(c)\geq \frac{4m}{2m+1}\theta_{n-k}+\frac{2m}{2m+1}.$$
\end{proof}

\begin{theorem} \label{th16}If $p\neq 2$, then $d(C_k(n,q)^\bot)\geq (4\theta_{n-k}+2)/3$, $q=p^h$, $p$ prime, $h\geq 1$.
\end{theorem}
\begin{proof} Let $c$ be a codeword of $C_k(n,q)^\bot$ with $wt(c)< (4\theta_{n-k}+2)/3$. According to Lemma \ref{lem12}, there is only one non-zero symbol used in $c$. Construct a $(k-1)$-space $\pi$ through a point $R$ of $supp(c)$ intersecting $supp(c)$ only in $R$. Then every $k$-space $K$ through $\pi$ has to contain at least $p-1$ extra points of $supp(c)$ in order to get $(c,K)=0$. But then $wt(c)\geq (p-1)\theta_{n-k}+1$, a contradiction.
\end{proof}
\begin{theorem} \label{th8}The minimum weight of $C_k(n,q)^\bot$ is at least $(12\theta_{n-k}+2)/7$ if $p=7$, and at least $(12\theta_{n-k}+6)/7$ if $p>7$.
\end{theorem}

\begin{proof}  We use the same techniques as in the proof of Proposition 2.4 in \cite{sachar}.
Let $c$ be a codeword of minimum weight of $C_k(n,q)^\bot$ and suppose that $wt(c)<(12\theta_{n-k}+6)/7$. It follows from Lemma \ref{lem12} that there are at most four different non-zero symbols used in the codeword $c$.
Suppose first that there are exactly two non-zero symbols used in $c$, say $1$ and $-1$. Suppose that the symbol $-1$ occurs the least, say $y$ times. Construct a $(k-1)$-space $\pi$ through a point $R$ of $supp(c)$, where $c_R=1$ and $\pi \cap supp(c)=\lbrace R \rbrace.$ Every $k$-space $\bar{\pi}$ through $\pi$ contains at least a second point of $supp(c)$. At most $y$ of those $k$-spaces contain a point $R'$ of $supp(c)$ with $c_{R'}=-1$, so at least $\theta_{n-k}-y$ of those $k$-spaces only contain points $R'$ of $supp(c)$ with $c_{R'}=1$. Since $(c,\bar{\pi})=0$, such $k$-spaces contain $0 \pmod{p}$ points of $supp(c)$. This yields
$$wt(c)\geq (\theta_{n-k}-y)(p-1)+y+1.$$
Using that $wt(c)<(12\theta_{n-k}+2)/7$ implies that 
$$p\theta_{n-k}-7\theta_{n-k}-p+7<0,$$
a contradiction if $p=7$. Using that $wt(c)<(12\theta_{n-k}+6)/7$ implies that
$$(p-7)\theta_{n-k}+7-3p<0,$$
a contradiction if $p>7$.

So we may assume that there are four non-zero symbols used in $c$, say $1,-1,a,-a$. Using the same notations as in the proof of Lemma \ref{lem12}, we see that
\begin{eqnarray}
wt(c)\geq 4X_R.
\end{eqnarray}
We call a $k$-space through one of the $(k-1)$-spaces $\bar{K}$, with $\bar{K}\cap supp(c)=\lbrace R \rbrace$, that has exactly two extra points of $supp(c)$, a \emph{ $3$-secant}. Let $X_3$ denote the number of $3$-secants through $\bar{K}$, and let $X_w$ denote the number of $k$-spaces through $\bar{K}$ that intersect $supp(c)$ in more than $3$ points. We have the following equations:
\begin{eqnarray}
wt(c)\geq 1+X_R+2X_3+3X_w,\\
\theta_{n-k}=X_R+X_3+X_w.
\end{eqnarray}
Suppose first that there are no $3$-secants, then substituting (3) in (1) and (2) gives
\begin{eqnarray}
wt(c)\geq 4\theta_{n-k}-4X_w,\\
wt(c)\geq 1+\theta_{n-k}+2X_w.
\end{eqnarray}
Eliminating $X_w$ using (4) and (5) gives
$$3wt(c)\geq 6\theta_{n-k}+2,$$
a contradiction.
This implies that $X_3\neq 0$. Let $T$ be a $3$-secant through $\bar{K}$. The sum of the symbols used in $T$ has to be zero, hence
\[(*)
\begin{array}{c}0=1+1+a \mbox{ and } a=-2 \mbox{, or}\\
0=1+a+a \mbox{ and } a=-1/2.\end{array}\]
For each point $P$ with $c_P=-a$, the $k$-space through $\bar{K}$ containing $P$ has to intersect $supp(c)$ in more than three points, since otherwise
\begin{eqnarray*}
1-a-a&=&0 \mbox{ and } a=1/2 \mbox{ or}\\
1+1-a&=&0 \mbox{ and } a=2.
\end{eqnarray*}
This contradicts (*) since $p>5$ implies that $\lbrace 2,-2\rbrace$ cannot be the same as $\lbrace 1/2,-1/2\rbrace$.
There are at least $X_R$ points with coefficient $-a$ and we see that they all must be on $k$-spaces contributing to $X_w$. Thus counting points again, we have
\begin{eqnarray}
wt(c)&\geq& 1+X_R+2X_3+X_R \nonumber \\
&=&1+2(\theta_{n-k}-X_3-X_w)+2X_3 \nonumber \\
&=&1+2\theta_{n-k}-2X_w.
\end{eqnarray}
Substituting (3) in (1) and (2) gives
\begin{eqnarray}
wt(c)&\geq& 4(\theta_{n-k}-X_3-X_w)\\
wt(c)&\geq& 1+\theta_{n-k}+X_3+2X_w.
\end{eqnarray}
Eliminating $X_3$ and $X_w$ using (6), (7) and (8) yields
$$7wt(c)\geq 12\theta_{n-k}+6$$
and the proof is complete.\end{proof}
The second part of the following theorem is Corollary 5.7.5 of \cite{AK}. Here we give an alternative proof, similar to \cite[Proposition 1]{Bagchi}.
\begin{theorem} \label{th5} 
(1) The only possible codewords of weight in $]\theta_k,(12\theta_{k}+6)/7[$ in $C_k(n,q)$, $k\geq n/2$, $q=p^h$, $p>7$ prime, $h\geq 1$, are scalar multiples of incidence vectors of non-linear blocking sets.

(2) The minimum weight of $C_k(n,q)$ is $\theta_k$, and a codeword of weight $\theta_k$ is a scalar multiple of the incidence vector of a $k$-space.
\end{theorem}
\begin{proof}(1) According to Lemma \ref{lem2}, there are two possibilities for a codeword $c\in C_k$ with $wt(c)<2q^k$. Either $(c,S)\neq 0$ for every $(n-k)$-dimensional space $S$, and Corollary \ref{cor2} yields that $c$ is a scalar multiple of the incidence vector of a non-linear blocking set,
or $(c,S)=0$ for all $(n-k)$-spaces $S$. But this implies that $c\in C_{n-k}^\bot$, which has weight at least $(12\theta_{k}+6)/7$ (see Theorem \ref{th8}).

(2) For the second statement, it is sufficient to use a result of Bose and Burton \cite{Bose} that shows that the minimum weight of a $k$-blocking set in $PG(n,q)$ is equal to $\theta_{k}$, and that this minimum is reached if and only if the blocking set is a $k$-space.
\end{proof}

\begin{remark}
In view of Theorem \ref{th5}, it is important to mention the conjectures made in \cite{sziklai}. If these conjectures are true (i.e. all small minimal blocking sets are linear), then Theorem \ref{th5} eliminates all codewords of $C_k(n,q)$ of weight in the interval $]\theta_k,(12\theta_{k}+6)/7 [$.
\end{remark}

In the cases $q=p$ and $q=p^2$, with $p$ a prime, we can deduce more. Theorem \ref{priem}, theorem \ref{th5}, Theorem \ref{th:1} and Theorem \ref{Weiner} yield the following theorems.  
\begin{theorem}\label{?} There are no codewords with weight in $]\theta_k, 2q^k[$ in $C_k(n,q)$, $k\geq n/2$, where $q=p$ is prime.
\end{theorem} 
\begin{theorem} \label{th20}
There are no codewords with weight in $]\theta_k, (12\theta_k+6)/7[$ in $C_k(n,q)$, $k\geq n/2$, where $q=p^2$, $p>11$ prime.
\end{theorem}

We now turn our attention to codewords in $C_k(n,q)$, $k\geq n/2$, $q=p^h$, $p$ prime, $h\geq 3$, with weight in $]\theta_k, (12\theta_k+6)/7[$. We know from Theorem \ref{th8} that such codewords belong to $C_k(n,q)\setminus C_k(n,q)^\perp$, so they define minimal $k$-blocking sets $B$ intersecting every $(n-k)$-dimensional space in $1\pmod{p}$ points (see Theorem \ref{blset}, Lemma \ref{lem5}). Let $e$ be the maximal integer for which $B$ intersects every $(n-k)$-space in $1 \pmod{p^e}$ points. 
In \cite[Corollary 5.2]{SF:07}, it is proven that 
\[|B|\geq q^k+\frac{q^k}{p^{e}+1}-1. \]
We now derive an upper bound on $|B|$, based on \cite[Theorem 5.3]{SF:07}.

\begin{theorem}\label{19} Let $B$ be a minimal  $k$-blocking set in
$PG(n,q)$, $n\geq 2$,  $q=p^h$, $p$ prime, $h\geq 1$, 
 intersecting every $(n-k)$-dimensional space in $1 \pmod{p^e}$ points,
with $e$ the maximal integer for which this is true. If $|B|\in ]\theta_k,(12\theta_k+6)/7[$ and that $p^e>2$, then
\[ |B|\leq   
q^k+\frac{2q^k}{p^e}.\]
\end{theorem}

\begin{proof} Put $E=p^e$ and let $\tau_{1+iE}$ be the number of $(n-k)$-dimensional spaces intersecting
$B$ in $1+iE$ points.
We count the number of $(n-k)$-dimensional spaces, the number of incident pairs
$(R,\pi)$, with $R \in B$ and with $\pi$ an $(n-k)$-dimensional space through $R$, and
the number of triples $(R,R',\pi)$, with $R$ and $R'$ distinct points
of $B$ and $\pi$ an $(n-k)$-dimensional space passing through $R$ and $R'$.
This gives us the following formulas.

\begin{eqnarray}
 \sum_{i \geq 0} \tau_{1+iE}&=&\frac{(q^{n+1}-1)(q^n-1)}{(q^{n-k+1}-1)(q^{n-k}-1)}
 \cdot X,\\
\sum_{i \geq 0} (1+iE)\tau_{1+iE} & = & |B|
\left(\frac{q^n-1}{q^{n-k}-1}\right)\cdot X,\\
\sum_{i \geq 0} (1+iE)(1+iE-1)\tau_{1+iE} &=&  |B|(|B|-1)\cdot X,
\end{eqnarray}
where \[X=\frac{(q^{n-1}-1)\cdots(q^{k+1}-1)}{(q^{n-k-1}-1)\cdots
(q-1)}
\]
is the number of $(n-k)$-dimensional spaces through a line of $PG(n,q)$.
Since $\sum_{i\geq 0} i(i-1) E^2 \tau_{1+iE} \geq 0$, we obtain
\[|B|(|B|-1) -(1+E)|B|\left(\frac{q^n-1}{q^{n-k}-1}\right)
+(1+E)\left(\frac{(q^{n+1}-1)(q^n-1)}{(q^{n-k+1}-1)(q^{n-k}-1)}\right)\geq 0.\]
Under the condition $2 <E$, this implies that \[|B| \leq 
q^k+\frac{2q^k}{E}.\]
%follows from the proof of Theorem
%\ref{thm5.2}. There we embedded the $t$-fold $(n-k)$-blocking set in
%$\pg(n,q)$ into $\pg(n,q^{n-k})$ to obtain a minimal $t$-fold 1-blocking
%set. We now apply the lower bound of Theorem \ref{thm3.1} on this $t$-fold 1-blocking set
%in $\pg(n,q^{n-k})$.
\end{proof}

\begin{remark}
If $p^e>4$, then $|B|<3/2q^k$ in which case results of Sziklai prove that $e$ is a divisor of $h$ \cite[Corollary 5.2]{sziklai}.
\end{remark}

We summarize the results on the minimum weight of $C_k(n,q)^\bot$, $k\geq n/2$, in the following table (with $\theta_n=(q^{n+1}-1)/(q-1)$).

\begin{center}
\begin{tabular}{|c|c|c|}\hline
$p$ & $h$ & $d$\\
\hline
$2$ & $(k,n)\neq (n-1,n)$ & $\theta_{n-k}+1<d\leq q^{n-k-1}(q+2)$\\
$p$ & $1$ & $2p^{n-k}$ \\
$2<p<7$ & $h>1$ & $ (4\theta_{n-k}+2)/3\leq d\leq 2q^{n-k}+\theta_{n-k-1}-\frac{q^{n-k}-1}{p-1}$\\
$7$ & $h>1$ & $(12\theta_{n-k}+2)/7 \leq d \leq 2q^{n-k}+\theta_{n-k-1}-\frac{q^{n-k}-1}{p-1}$\\
$p>7$ & $h>1$ &   $(12\theta_{n-k}+6)/7\leq d\leq 2q^{n-k}+\theta_{n-k-1}-\frac{q^{n-k}-1}{p-1}$\\ \hline
\end{tabular}

\vspace{0,3 cm}

{\sl Table 1:  The minimum weight $d$ of $C_k(n,q)^\bot$, $k\geq n/2$, $q=p^h$, $p$ prime, $h\geq 1$}
\end{center}

Address of the authors:\\

 \hspace{-1cm}Ghent University, Dept. of Pure Mathematics and Computer Algebra, Krijgslaan 281-S22, 9000 Ghent, Belgium\\
 
\hspace{-1cm} \begin{tabular}{lll}
 Michel Lavrauw: & ml@cage.ugent.be & http://cage.ugent.be/$\sim$ml\\
Leo Storme: & ls@cage.ugent.be&  http://cage.ugent.be/$\sim$ls\\
Geertrui Van de Voorde: & gvdvoorde@cage.ugent.be &http://cage.ugent.be/$\sim$gvdvoorde \\
\end{tabular}

\end{document}